\documentclass[10pt]{amsart}  

\usepackage{amsmath,amssymb,amsfonts,bbm,color,hyperref}
\usepackage{algorithm}
\usepackage{algpseudocode}

\title{Covering compact metric spaces greedily}

\author{Jan Hendrik Rolfes} 
\address{J.H.~Rolfes, Mathematisches Institut, Universit\"at zu
 K\"oln, Weyertal~86--90, 50931 K\"oln, Germany}
\email{j.rolfes@uni-koeln.de}

\author{Frank Vallentin} 
\address{F.~Vallentin, Mathematisches Institut, Universit\"at zu
 K\"oln, Weyertal~86--90, 50931 K\"oln, Germany}
\email{frank.vallentin@uni-koeln.de}

\thanks{The second author is partially supported by the
SFB/TRR 191 ``Symplectic Structures in Geometry,
Algebra and Dynamics'', funded by the DFG}

\date{January 22, 2018}



\newcommand{\R}{\mathbb{R}}

\newtheorem{defin}{Definition}[section]

\newtheorem{theorem}[defin]{Theorem}

\newtheorem{corollary}[defin]{Corollary}
\newtheorem{lemma}[defin]{Lemma}

\begin{document}

\begin{abstract} 
  A general greedy approach to construct coverings of compact metric
  spaces by metric balls is given and analyzed. The analysis is a
  continuous version of Chv\'atal's analysis of the greedy algorithm
  for the weighted set cover problem. The approach is demonstrated in
  an exemplary manner to construct efficient coverings of the
  $n$-dimensional sphere and $n$-dimensional Euclidean space to give
  short and transparent proofs of several best known bounds obtained
  from deterministic constructions in the literature on sphere
  coverings.
\end{abstract}

\maketitle

\markboth{J.H.~Rolfes and F.~Vallentin}{Covering compact metric spaces greedily}

\section{Introduction}

Let $X$ be a compact metric space having metric $d$. Given a scalar
$r\in\R_{\geq 0}$ we define the \emph{closed ball} of radius $r$
around center $x \in X$ by
\[
B(x,r) = \{ y \in X : d(x,y) \leq r\}.
\]

The \emph{covering number} of the space $X$ and a positive number $r$
is
\[
\mathcal{N}(X,r) = \min \left\{ |Y| :  Y \subseteq X,\, \bigcup_{y \in Y} B(y,r) = X \right\},
\]
i.e.\ it is the smallest number of balls with radius $r$ one needs to
cover~$X$. Determining the covering number is a fundamental problem in
metric geometry (see for example the classical book by Rogers
\cite{Rogers1964a}) with many applications: compressive sensing
\cite{Foucart2013a}, approximation theory and machine learning
\cite{Cucker2002a} --- to name a few.

\smallskip

In this paper we are concerned with compact metric spaces which carry
a probability measure $\omega$; a Borel measure normalized by
$\omega(X) = 1$. We will assume that this probability measure behaves
homogeneously on balls and is non degenerate, i.e.\ it satisfies the
following two conditions:
\begin{enumerate}
\item[(a)] $\omega(B(x,s))=\omega(B(y,s))$ for all $x,y\in X$, and for
 all $s \geq 0$,
\item[(b)] $\omega(B(x,\varepsilon))>0$ for all $x\in X$, and for all $\varepsilon>0$.
\end{enumerate}

By (a) the measure of a ball does only depend on the radius $s$ and not
on the center $x$, so we simply denote $\omega(B(x,s))$ by $\omega_s$
throughout the paper.

\begin{theorem}
\label{thm:main}
Let $(X,d)$ be a compact metric space with probability measure
$\omega$ satisfying conditions~(a) and~(b).  Then for every
$\varepsilon$ with $r/2 > \varepsilon > 0$ the covering number satisfies
\[
\frac{1}{\omega_r} \leq \mathcal{N}(X,r) \leq
\frac{1}{\omega_{r-\varepsilon}}
\left(\ln\left(\frac{\omega_{r-\varepsilon}}{\omega_\varepsilon}\right)+1\right).
\]
\end{theorem}

The lower bound is obvious (using the $\sigma$-subadditivity of
$\omega$). We give a proof for the upper bound in
Section~\ref{sec:proof}. Our proof is based on a greedy approach to
covering. We iteratively choose balls which cover the maximum measure
of yet uncovered space.

This greedy algorithm has been analyzed in the finite setting of the
\textsc{set cover} problem which is a fundamental problem in
combinatorial optimization. The \textsc{set cover} problem is defined
as follows. Given a collection $S_1, \ldots, S_m$ of the ground set
$\{1, \ldots, n\}$ and given costs $c_1, \ldots, c_m$ the task is find
a set of indices $I \subseteq \{1, \ldots, m\}$ such that
$\bigcup_{i \in I} S_i = \{1, \ldots, n\}$ and $\sum_{i \in I} c_i$ is
as small as possible.

Computationally, the \textsc{set cover} problem is difficult; Dinur
and Steurer \cite{Dinur2014a} showed that for every $\varepsilon > 0$
it is $\mathrm{NP}$-hard to find an approximation to the \textsc{set
  cover} problem within a factor of $(1-\varepsilon) \ln n$.

On the other hand, Chv\'atal \cite{Chvatal1979a} (previously, Johnson
\cite{Johnson1974a}, Stein \cite{Stein1974a} and Lov\'asz
\cite{Lovasz1975a} proved similar results for the case of uniform
costs $c_1 = \ldots = c_m = 1$) showed that the greedy algorithm gives
an $(\ln n + 1)$-approximation for the \textsc{set cover}
problem. More specifically, Chv\'atal showed that the natural linear
programming relaxation of \textsc{set cover}
\[
\begin{split}
\text{minimize } & \;\sum_{i=1}^m c_i x_i\\
\text{subject to } & \; x_1, \ldots, x_m \geq 0\\
& \; \sum_{i : j \in S_i} x_i \geq 1 \text{ for all } j = 1, \ldots, n
\end{split}
\]
is at most a factor of
$H_k = \sum_{n=1}^k \frac{1}{n} \leq \ln k + 1$, with
$k = \max_i |S_i|$, away from an optimal solution of \textsc{set
  cover}. He proved this bound by exhibiting an appropriate feasible
solution of the dual of the linear programming relaxation. The greedy
algorithm is used to construct this feasible solution.

In Section~\ref{sec:proof} we transfer Chv\'atal's argument from the
finite \textsc{set cover} setting to the setting of compact metric
spaces. Function $g$ appearing there features the feasible solution of
the dual linear program. This will provide a proof of
Theorem~\ref{thm:main}. In Section~\ref{sec:applications} we apply
Theorem~\ref{thm:main} to three concrete geometric settings and we
retrieve some of the best known asymptotic results, unifying many
results on sphere coverings.

We think that the $\mathrm{NP}$-hardness of getting
$(1-\varepsilon)\ln n$-approximations for the \textsc{set cover}
problem is a natural barrier for getting better asymptotic results for
geometric covering problems. This might serve as an explanation why
progress for example on the sphere covering problem has been very slow
since the initial work of Rogers \cite{Rogers1964a}.

We are not the first observing the strong relation between geometric
covering problems and \textsc{set cover}\footnote{In fact, we realized
  this only after we, in an attempt to understand geometric covering
  problems from an optimization point of view, wrote down the main
  body of this paper.}. In recent papers, Artstein-Avidan and Raz
\cite{Artstein2011a}, Artstein-Avidan and Slomka \cite{Artstein2015a}
and especially Nasz\'odi \cite{Naszodi2016a} used the results of
Lov\'asz \cite{Lovasz1975a} to unify old results and prove new results
on geometric coverings. However, they apply the results from
\textsc{set cover} directly after choosing a finite
$\varepsilon$-net. Since we consider an infinite analogue of
\textsc{set cover} we do not need to use an $\varepsilon$-net and by
this we sometimes get slightly better constants and more importantly
we think that the analysis becomes rather beautiful.

Using the relation between geometric covering problems and \textsc{set
  cover} has already turned out to be fruitful: Prosanov
\cite{Prosanov2017a} found new upper bounds for the chromatic number
of distance graphs on the unit sphere, Nasz\'odi and Polyanskii
\cite{Naszodi2017a} studied multi covers by this approach.

\section{Proof of Theorem~\ref{thm:main}}
\label{sec:proof}

We shall prove that the following greedy algorithm
(Algorithm~\ref{algo:Chvcover}) will provide a covering of $X$ with at
most
\[
\frac{1}{\omega_{r-\varepsilon}}
\left(\ln\left(\frac{\omega_{r-\varepsilon}}{\omega_\varepsilon}\right)+1\right)
\]
many balls of radius $r$.

\begin{algorithm}[h]
	\caption{Greedy algorithm}\label{algo:Chvcover}
	\begin{algorithmic}[1]
\State $i \gets 0$
\State $S_x^i = B(x,r-\varepsilon)$ for all $x\in X$
\While {$\bigcup_{j=1}^i B(y^j,r) \neq X$} 
   \State $i \gets i+1$ 
	\State Choose $y\in X$ with $\omega(S_y^{i-1})\geq \omega(S_x^{i-1})$ for all $x\in X$
	\State $y^i = y$ 
	\State $S_x^i = S_x^{i-1}\setminus S_y^{i-1}$ for all $x\in X$
\EndWhile
\end{algorithmic}
\end{algorithm}

We split the proof into three lemmas where the following identity will
become important:
\begin{equation}
\label{eq:fundamental}
S_x^{i-1} = B(x,r-\varepsilon) \setminus \bigcup_{j=1}^{i-1} B(y^j, r-\varepsilon).
\end{equation}

\smallskip

The first lemma states that the step of the algorithm when we want to
choose $y \in X$, with $\omega(S^{i-1}_y) \geq \omega(S^{i-1}_x)$ for
all $x \in X$, is indeed well-defined.

\begin{lemma}
\label{lem:welldefined}
In every iteration $i$ the supremum $\sup\{\omega(S_x^{i-1}) : x \in X\}$
is attained.
\end{lemma}

\begin{proof}
 We shall show that the function $f_i \colon X \to \R$,
 $f_i(x)=\omega(S_x^{i-1})$ is continuous for every iteration
 $i$. This implies that $f_i$ attains its maximum since $X$ is
 compact.

For $x, y \in X$ we have
\[
\begin{split} |f_i(x)-f_i(y)| 
& = |\omega(S_x^{i-1}) - \omega(S_y^{i-1})|\\
& = |\omega(S_x^{i-1} \setminus S_y^{i-1}) + \omega(S_x^{i-1}
 \cap S_y^{i-1})\\
& \qquad - (\omega(S_y^{i-1} \setminus S_x^{i-1}) +
 \omega(S_y^{i-1} \cap S_x^{i-1}))|\\
& = |\omega(S_x^{i-1} \setminus S_y^{i-1}) - \omega(S_y^{i-1}\setminus S_x^{i-1})|\\
& \leq \max\{\omega(S_x^{i-1}\setminus S_y^{i-1}), \omega(S_y^{i-1}\setminus S_x^{i-1})\}.
\end{split}
\]
Without loss of generality, the maximum is attained
at $\omega(S_x^{i-1}\setminus S_y^{i-1})$. Then by \eqref{eq:fundamental} we see
\[
S_x^{i-1}\setminus S_y^{i-1} \subseteq B(x,r - \varepsilon) \setminus B(y,r -
\varepsilon).
\]
By the triangle inequality
\[
B(x,r - \varepsilon) \setminus B(y,r -
\varepsilon) \subseteq B(y,r - \varepsilon +d(x,y))\setminus B(y,r - \varepsilon).
\]
Now consider the indicator function
$\mathbbm{1}_{B(y,r - \varepsilon +d(x,y))\setminus B(y,r -
 \varepsilon)}$.
When $y$ tends to $x$, then we have a monotonously decreasing sequence
of measurable functions tending to $0$. By applying the theorem of
monotone convergence we obtain that the integral
\[
\int \mathbbm{1}_{B(y,r - \varepsilon +d(x,y))\setminus B(y,r -
\varepsilon)}(z)\, d\omega(z)
\]
tends to $0$ as well. Hence, $f_i(y)$ tends to $f_i(x)$.
\end{proof}

The second lemma states that the algorithm terminates after finitely
many iterations.

\begin{lemma}\label{lem:termination}
 Algorithm \ref{algo:Chvcover} terminates after at most
 $\omega_\varepsilon^{-1}$ iterations and returns a covering.
\end{lemma}

\begin{proof}
 Consider the $i$-th iteration of the algorithm and suppose there
 exists $z\in X$ with $z\notin \bigcup_{j=1}^{i-1} B(y^j,r)$. From
 the triangle inequality it follows that
\[
B(z,\varepsilon)\cap B(y^j,r-\varepsilon) = \emptyset.
\]
Together with \eqref{eq:fundamental} it implies that
$B(z,\varepsilon)\subseteq S_z^{i-1}$. Choose $y \in X$ with
$\omega (S^{i-1}_y) \geq \omega (S^{i-1}_x)$ for every $x \in X$. Hence
we have
\[
\omega(S^{i-1}_y) \geq \omega (S^{i-1}_z) \geq
\omega(B(z,\varepsilon)) = \omega_\varepsilon >0,
\]
where $\omega_\varepsilon$ is positive by assumption (b) and thus
\[
1 = \omega(X) \geq \sum_{j=1}^i\omega(S_{y^j}^{j-1})\geq
i\cdot\omega_\varepsilon,
\]
where the first inequality follows because the sets
$S_{y^j}^{j-1}$, with $j = 1, \ldots, i$, are pairwise
disjoint.  So after at most $\omega_\varepsilon^{-1}$
iterations, the algorithm terminates with a covering.
\end{proof}

The third lemma gives the desired upper bound for the covering number.

\begin{lemma}
Algorithm~\ref{algo:Chvcover} terminates after at most 
\[
\frac{1}{\omega_{r-\varepsilon}}
\left(\ln\left(\frac{\omega_{r-\varepsilon}}{\omega_{\varepsilon}}\right)
  + 1\right)
\]
iterations. In particular, this number gives an upper bound for the
covering number $\mathcal{N}(X,r)$.
\end{lemma}

\begin{proof}
 Let $Y \subseteq X$ denote the covering produced by Algorithm
 \ref{algo:Chvcover} after $|Y|$ iterations.  We shall prove
\begin{equation}
\label{densitystatementm}
\ln\left(\frac{\omega_{r - \varepsilon}}{\omega_\varepsilon}\right)+ 1  \geq |Y|\cdot \omega_{r - \varepsilon}. 
\end{equation}

For this we define the symmetric kernel $K \colon X\times X \to \mathbb{R}$ by
\[
K(x,y)=
\begin{cases}
1, & \text{if }  y \in B(x,r - \varepsilon)\\
0, & \text{otherwise.}
\end{cases}
\]
For every $x\in X$ the following equality
\[
\int K(x,y) \, d\omega(y) = \omega_{r - \varepsilon}
\]
holds because for every fixed $x \in X$ we have
$K(x,y)=\mathbbm{1}_{B(x,r-\varepsilon)}(y)$ for all $y \in X$. We
will exhibit an integrable function
$g \colon X \rightarrow \mathbb{R}$ satisfying
\begin{equation}
\label{help1m}
\int K(x,y) g(x) \, d\omega(x) \leq \ln\left(\frac{\omega_{r -
   \varepsilon}}{\omega_\varepsilon}\right)+1
\end{equation}
for all $y\in X$ and satisfying
\begin{equation} 
\label{help2m}
\int g(x)\, d\omega(x) = |Y|.
\end{equation}
Combining \eqref{help1m} and \eqref{help2m}, we get
\[
\begin{split}
\ln\left(\frac{\omega_{r-\varepsilon}}{\omega_{\varepsilon}}\right)
 + 1
&  \geq \int \int K(x,y) g(x) \, d\omega(x) d\omega(y)\\
& = \int  g(x) \int K(x,y) \, d\omega(y) d\omega(x) \\
& = \int  g(x) \omega_{r - \varepsilon} \, d\omega(x) \\
& = |Y| \cdot \omega_{r - \varepsilon}
\end{split}
\]
and we have proven~\eqref{densitystatementm}. 

\bigskip

Now we only have to exhibit the function $g$.

\smallskip

For brevity, we denote $\omega^{i-1}_y=\omega(S_y^{i-1})$. We define $g$ as follows:
\[
g(x) =
\begin{cases} 
(\omega_{y^i}^{i-1})^{-1}, & \text{ if } x\in S_{y^i}^{i-1},\\ 
0, & \text{otherwise,}
\end{cases}
\]
which is a valid definition since the sets $S_{y^i}^{i-1}$ are
pairwise disjoint. Also observe that $g$ is an integrable function on
the compact set $X$.

From this definition of~$g$ we immediately get \eqref{help2m}:
\[
\int g(x)\, d\omega(x) = \sum_{i=1}^{|Y|}
\omega_{y^i}^{i-1} (\omega_{y^i}^{i-1})^{-1} = |Y|,
\]

To prove \eqref{help1m} we fix $y \in X$. We observe the equality
\[
B(y,r - \varepsilon)\cap S_{y^i}^{i-1} = S_y^{i-1}\setminus S_y^i,
\]
which describes which part of $B(y,r - \varepsilon)$ is cut away in iteration
$i$. Then,
\[
\begin{split}
\int K(x,y) g(x) \, d\omega(x) 
& = \sum_{i=1}^{|Y|} \int K(x,y) \mathbbm{1}_{S_{y^i}^{i-1} }(x) (\omega^{i-1}_{y^i})^{-1} \, d\omega(x)\\
& = \sum_{i=1}^{|Y|} \int \mathbbm{1}_{S_y^{i-1}\setminus S_y^i}(x) (\omega^{i-1}_{y^i})^{-1} \, d\omega(x)\\
& = \sum_{i=1}^{|Y|} (\omega^{i-1}_y - \omega^i_y) (\omega^{i-1}_{y^i})^{-1}
\end{split}
\]
For $y\in X$ consider the last iteration $b$ such that
\begin{equation}
\label{help4m}
\omega_{r-\varepsilon} = \omega(B(y,r - \varepsilon)) = \omega_y^0\geq \omega_y^1 \geq \ldots \geq \omega_y^b
\geq \omega(B(y, \varepsilon)) = \omega_{\varepsilon}
\end{equation}
holds (here we used $r/2 > \varepsilon$). Note that $b < |Y|$. Note also that $\omega_y^{i-1} \leq
\omega^{i-1}_{y^i}$ holds. We split the sum above into two parts:
\[
\begin{split}
\sum_{i=1}^{|Y|} (\omega^{i-1}_y - \omega^i_y)
(\omega^{i-1}_{y^i})^{-1}
& = \sum_{i=1}^{b} (\omega^{i-1}_y - \omega^i_y)
(\omega^{i-1}_{y^i})^{-1}
+ \sum_{i=b+1}^{|Y|} (\omega^{i-1}_y - \omega^i_y)
(\omega^{i-1}_{y^i})^{-1}\\
& \leq \sum_{i=1}^{b} (\omega^{i-1}_y - \omega^i_y)
(\omega^{i-1}_{y})^{-1}
+ (\omega^b_y - \omega^{b+1}_y) (\omega^b_y)^{-1} \\
& \qquad + \sum_{i =
 b+2}^{|Y|} (\omega^{i-1}_y - \omega^i_y)
\omega_{\varepsilon}^{-1}\\
& \leq \left(\sum_{i=1}^{b} (\omega^{i-1}_y - \omega^i_y)
(\omega^{i-1}_{y})^{-1} + \frac{\omega^b_y -
 \omega_\varepsilon}{\omega^b_y} \right) \\
& \qquad + \left(\frac{\omega_{\varepsilon} -
   \omega^{b+1}_y}{\omega_{\varepsilon}} + \frac{\omega^{b+1}_y - \omega^{|Y|}_y}{\omega_{\varepsilon}}\right).
\end{split}
\]
The first sum is a lower Riemann sum of the function $x \mapsto \frac{1}{x}$ in
the interval $[\omega_\varepsilon, \omega_{r-\varepsilon}]$
and thus we have $\ln\left(\frac{\omega_{r -
   \varepsilon}}{\omega_{\varepsilon}}\right)$ as an
upper bound. The second sum is clearly bounded above by $1$. Hence, \eqref{help1m} holds.
\end{proof}

\section{Applications of Theorem~\ref{thm:main}}
\label{sec:applications}

\subsection{Covering the $n$-dimensional sphere}
\label{ssec:sphere}

As a first application of Theorem~\ref{thm:main} we consider the
problem of covering the $n$-dimensional sphere
\[
X = S^n = \{x \in \mathbb{R}^{n+1} : x \cdot x = 1\},
\]
equipped with spherical distance
\[
d(x,y) = \arccos x \cdot y \in [0,\pi]
\]
and with the rotationally invariant probability measure $\omega$, by
spherical caps / metric balls $B(x,r)$. Clearly, properties (a) and
(b) are satisfied in this setting. Again we set
$\omega_r = \omega(B(x,r))$.

We are especially interested in the covering number
$\mathcal{N}(S^n, r)$ when $0 < r < \pi/2$ or equivalently in the
covering density defined by $\omega_r \cdot \mathcal{N}(S^n,
r)$. Theorem~\ref{thm:main} says that the covering density is at most
\begin{equation}
\label{eq:density-bound}
\frac{\omega_r}{\omega_{r-\varepsilon}} \left(
\ln\left(\frac{\omega_{r-\varepsilon}}{\omega_{\varepsilon}}\right) + 1
\right).
\end{equation}
This upper bounds holds for every $\varepsilon$ with
$0 < \varepsilon < r$. By choosing $\varepsilon$ depending on the
dimension $n$ and on the spherical distance $r$ we can find an upper
bound for the covering density which only depends on $n$.

For this we recall a useful estimate of fractions of the form
$\omega_{tr}/\omega_r$ due to B\"or\"ozky~Jr.\ and Wintsche
\cite{Böröczky2003a}:
\begin{equation}
\label{eq:bw-bound}
\frac{\omega_{tr}}{\omega_r} \leq t^n \quad \text{whenever } r < tr < \frac{\pi}{2}.
\end{equation}
We set $\varepsilon = r/(\mu n +1)$ with parameter $\mu > 1$ which we
are going to adjust later. Furthermore, we set
\[
t = \frac{r}{r-\varepsilon} = 1 + \frac{1}{\mu n}
\]
and
\[
t' =\frac{r-\varepsilon}{\varepsilon} = \mu n.
\]
By using \eqref{eq:density-bound} and \eqref{eq:bw-bound} we have the
following upper bound for the covering density
\[
\begin{split}
\frac{\omega_r}{\omega_{r-\varepsilon}} \left(
\ln\left(\frac{\omega_{r-\varepsilon}}{\omega_{\varepsilon}}\right) + 1
\right) 
& \leq \left( 1 + \frac{1}{\mu n}\right)^n (n \ln \mu n + 1)\\
& \leq e^{1/\mu} (n \ln \mu n + 1)\\
& \leq \left(1 + \frac{1}{\mu - 1}\right)  (n \ln \mu n + 1)
\end{split}
\]

Thus we have proven:

\begin{corollary}
 The covering density of the $n$-dimensional sphere by spherical
 balls is at most
\[
\left(1 + \frac{1}{\mu - 1}\right)  (n \ln \mu n + 1) \text{ for all }
\mu > 1.
\]
In particular, for $\mu = \ln n$, the covering density is at most
\[
n \ln n + n \ln\ln n + n + o(n).
\]
\end{corollary}

In the asymptotic case the best known bound is $(1/2 + o(1)) n \ln n$
due to Dumer \cite{Dumer2007a} which comes from a randomized
construction. Our corollary slightly improves the previously best
known non-asymptotic bound $n \ln n + n \ln\ln n + 2n + o(n)$ by
B\"or\"ozky~Jr.\ and Wintsche \cite{Böröczky2003a} also coming from a
randomized construction.

\subsection{Covering $n$-dimensional Euclidean space}
\label{ssec:space}

As a second application we consider coverings of $n$-dimensional
Euclidean space~$\mathbb{R}^n$ by congruent balls. We get a covering
of $\mathbb{R}^n$ by applying Theorem~\ref{thm:main} to the torus
$\mathbb{T}^n = \mathbb{R}^n / \mathbb{Z}^n$ which is a compact metric
space satisfying properties (a) and (b). Then we periodically extend
the obtained covering of $\mathbb{T}^n$ to a covering of the entire
$\mathbb{R}^n$ having the same covering density.

We repeat the choices and calculations as in the previous section
(which are slightly simpler here because clearly
$\omega_{tr}/\omega_r = t^n$ holds where here $\omega$ denotes the
Lebesgues measure) and get:

\begin{corollary}
 The covering density of the $n$-dimensional Euclidean space
 by congruent balls is at most
\[
\left(1 + \frac{1}{\mu - 1}\right)  (n \ln \mu n + 1) \text{ for all }
\mu > 1.
\]
In particular, for $\mu = \ln n$, the covering density is at most
\[
n \ln n + n \ln\ln n + n + o(n).
\]
\end{corollary}

We remark that this bound coincides with the currently best known
bound by G. F\'ejes Toth \cite{FejesToth2009a} coming from a
deterministic construction. The best known bound coming from a
randomized construction is $(1/2 + o(1)) n \ln n$ due to Dumer
\cite{Dumer2007a}

\subsection{More general coverings}
\label{ssec:general}

At last we want to demonstrate that the greedy approach to geometric
covering problems is quite flexible. It is not restricted to finding
coverings of compact metric spaces by balls but can be extended to
finding coverings of compact metric spaces by finite unions of balls
\[
\bigcup_{i=1}^N B(y_i,r),
\]
where we choose the initial points $y_1,\ldots , y_N\in X$ arbitrarily. 

We make this statement precise in the general setting of a compact
metric space $(X,d)$. Consider the group of continuous isometries of
$(X,d)$, these are all continuous bijective maps $\tau \colon X \to X$
which preserve the distance between every two points $x,y \in X$. We
assume that the group acts transitively on $X$ and that
$\omega(\tau A) = \omega(A)$ holds for all continuous isometries
$\tau$ and all measurable sets $A$. Then by the theorem of
Arzel\`a-Ascoli (see for example \cite[Chapter 4.6]{Folland1999a}) the
group of continuous isometries is relatively compact in the compact
space of continuous maps mapping $X$ to itself equipped with the
supremum norm. We need this compactness for Lemma
\ref{lem:welldefined}. So we can transfer the analysis of the greedy
algorithm given in Section~\ref{sec:proof} to this setting.

With small modifications this extension can for example be applied to
prove the following theorem due to Nasz\'odi \cite[Theorem
1.3]{Naszodi2016a}:

\begin{theorem}
\label{Naszth}
Let $K\subseteq \R^n$ be a bounded measurable set. Then there is a
covering of $\R^n$ by translated copies of $K$ of density at most
\[
\inf\left\{
\frac{\omega(K)}{\omega(K_{-\delta})}\left(
   \ln\left(\frac{\omega\left(K_{-\delta/2}
     \right)}{\omega(B(0,\delta/2))}\right)  +1\right) :
\delta > 0, K_{-\delta} \neq \emptyset\right\},
\]
where $K_{-\delta} = \{x\in K :\ B(x,\delta)\subseteq K\}$ is the $\delta$-inner parallel body of $K$.
\end{theorem}

Here, we only sketch the proof, though filling in the details is
easy. As in Section~\ref{ssec:space} we can work on the torus
$\mathbb{T}^n$. We approximate the body $K$ and its inner parallel
bodies by a finite union of balls for which
\[
\bigcup_{i=1}^N B(y_i, \delta) \subseteq K
\]
and
\[
K_{-\delta} \subseteq \bigcup_{i=1}^N B(y_i, \delta/2) \subseteq K_{-\delta/2}
\]
holds. In the end going back from the torus $\mathbb{T}^n$ to $\mathbb{R}^n$
we get a covering of $\mathbb{R}^n$ by translated copies of $K$ with
density at most
\[
\inf\left\{
\frac{\omega(K)}{\omega(K_{-\delta/2})}\left(
   \ln\left(\frac{\omega\left(K_{-\delta/2}
     \right)}{\omega(B(0,\delta/2))}\right)  +1\right) :
\delta > 0, K_{-\delta} \neq \emptyset\right\},
\]
improving the result of Nasz\'odi slightly.

Another alternative of proving this bound is to verify that the proof
of Theorem \ref{thm:main} also holds if we consider translates of
$K_{-\delta/2}\subseteq \R^n$ instead of balls
$B(x,r-\varepsilon)$. This further requires that $K_{-\delta}$ is
nonempty and to consider translates of Minkowski sums
$x + K_{-\delta/2} + B(0, \zeta)$ instead of
$B(x, r - \varepsilon + \zeta)$ in the parts of the proofs of Lemmas
\ref{lem:welldefined} and \ref{lem:termination} where we apply the
triangle inequality.

\section*{Acknowledgements}

We thank Markus Schweighofer, Cordian Riener, and the anonymous referee for helpful remarks.

\end{document}